\documentclass[11pt]{amsart}
\usepackage[T1]{fontenc}
\usepackage{lmodern}
\usepackage{amsmath,amssymb,amsthm,mathtools}
\usepackage[margin=1in]{geometry}
\usepackage{microtype}
\usepackage[colorlinks=true,linkcolor=blue,citecolor=blue,urlcolor=blue]{hyperref}
\usepackage{comment}

\newtheorem{theorem}{Theorem}[section]
\newtheorem{proposition}[theorem]{Proposition}
\newtheorem{lemma}[theorem]{Lemma}
\newtheorem{corollary}[theorem]{Corollary}
\theoremstyle{definition}
\newtheorem{definition}[theorem]{Definition}
\newtheorem{example}[theorem]{Example}
\theoremstyle{remark}
\newtheorem{remark}[theorem]{Remark}

\newcommand{\R}{\mathbb R}
\newcommand{\C}{\mathbb C}
\newcommand{\K}{\mathbb K}
\newcommand{\T}{\mathbb T}

\newcommand{\norm}[1]{\left\lVert#1\right\rVert}
\newcommand{\abs}[1]{\left|#1\right|}
\newcommand{\diamset}{\mathcal D}
\newcommand{\Lip}{\operatorname{Lip}}

\newcommand{\AC}{\operatorname{AC}}

\title[Extension of isometries between unit spheres]
{Extension of surjective isometries between unit spheres of $\ell^p$-direct sums}
\author{Min-Ruei Lin}
\address{Department of Applied Mathematics, National Sun Yat-sen University, Kaohsiung 80424, Taiwan; and Graduate School of Science and Technology, Niigata University, Niigata 950--2181, Japan}
\email{m082030021@student.nsysu.edu.tw}
\subjclass[2020]{Primary 46B04; Secondary 46B20, 46E15, 46E40}
\keywords{Tingley's problem, direct sum, strictly convex space, absolutely continuous function}
\begin{document}

\begin{abstract}
We study surjective isometries between the unit spheres of Banach spaces of the form $A\oplus_p B$, where $A$ is nonzero and strictly convex, $1<p<\infty$, and every unit vector of $B\neq\{0\}$ has at least two distinct unit vectors at distance $2$.
For two direct sums $A_i\oplus_{p_i}B_i$, $i=1,2$, satisfying these hypotheses, we prove that every surjective isometry 
$$\Delta:S_{A_1\oplus_{p_1}B_1}\to S_{A_2\oplus_{p_2} B_2}$$
between their unit spheres forces equality of the direct-sum exponents $p_i$ and extends uniquely to a surjective real-linear isometry
$$\widetilde{\Delta}:A_1\oplus_p B_1\to A_2\oplus_p B_2$$
preserving the corresponding components, where $p:=p_1=p_2$.
The proof identifies the component unit spheres through diametral sets and distance formulas, and then constructs isometries between the component unit balls. 

As a principal application, we obtain extension and representation results for absolutely continuous functions equipped with norms combining evaluation at a fixed point and the $L^1$-norm of the derivative. 
We also discuss Lipschitz and differentiable function spaces, distinguishing the consequences of our theorem from those already covered by existing extension results. 
Endpoint conclusions are derived separately, including the Mazur--Ulam property for the real absolutely continuous space with the sum norm.
\end{abstract}
\maketitle

\section{Introduction}
In 1987, Tingley~\cite{Tingley} formulated the following problem:
if $X$ and $Y$ are Banach spaces and $\Delta:S_X\to S_Y$ is a surjective isometry between their unit spheres, does $\Delta$ extend to a surjective real-linear isometry from $X$ onto $Y$?
Here $S_X$ denotes the unit sphere of $X$.
Tingley's problem has received considerable attention, and positive solutions have been obtained for many classical Banach spaces and spaces arising in operator algebra theory.
For example, Wang~\cite{WangContinuous} established an extension theorem for scalar-valued continuous-function spaces, and Tan~\cite{Tan} proved that real $L^1$-spaces have the Mazur--Ulam property.
Mori and Ozawa \cite{MoriOzawa} proved the Mazur--Ulam property for unital $C^*$-algebras, meaning that the target space may be arbitrary.
Banakh \cite{Banakh} resolved the problem for two-dimensional Banach spaces.
Further developments include the Mazur--Ulam property for unital $JB^*$-algebras~\cite{PeraltaSvarc}.
We refer to~\cite{PeraltaSurvey} for a survey of developments concerning operator algebras.

Despite these advances, Tingley's problem remains open in general, including for arbitrary finite-dimensional Banach spaces of dimension at least three. 
The known positive results therefore continue to motivate the search for geometric hypotheses under which the behavior of a surjective isometry on the unit sphere forces linearity.
For direct sums, a basic issue is whether the isometry recognizes the component spaces.
Once the component unit spheres have been identified, one can ask whether their images determine isometries between the component unit
balls and hence linear extensions.

Direct sums of continuous-function spaces provide an important setting for this approach.
Wang and Orihara~\cite{WangOrihara} studied surjective isometries between the unit spheres of $\ell^1$-sums of $C_0(\Omega,E)$-type spaces.
Wang~\cite{WangLp} subsequently established extension and representation theorems for $\ell^p$-sums of $C_0(\Omega,E)$-type spaces, where $1<p<\infty$.
In particular, \cite[Theorem~1.8]{WangLp} applies to
\[
 A\oplus_p C(K,E),
\]
where $A$ is a nonzero strictly convex Banach space, $K$ is a compact Hausdorff space containing at least two points, and $E$ is an arbitrary nonzero Banach space.
In fact, the first component can be represented as $C_0(\{*\},A)$.
Thus the freedom to choose an arbitrary strictly convex first component is already present in Wang's theorem.

The present paper considers a geometric condition \textup{(D)} on the second component that does not require a specified continuous-function representation.
Condition \textup{(D)} is defined in the next section.
We state the main theorem as follows.

\begin{theorem}\label{thm:main}
For $i=1,2$, let $A_i$ be a nonzero strictly convex Banach space, let $B_i$ be a nonzero Banach space satisfying condition \textup{(D)}, and let $1<p_i<\infty$.
Then the following statements are equivalent:
\begin{enumerate}
\item There exists a surjective isometry $\Delta:S_{A_1\oplus_{p_1} B_1}\to S_{A_2\oplus_{p_2} B_2}$.
\item $p_1=p_2=:p$, and there exist surjective real-linear isometries $U:A_1\to A_2$ and $V:B_1\to B_2$.
\end{enumerate}
Moreover, for each such $\Delta$, there are unique
surjective real-linear isometries $U$ and $V$ satisfying
\begin{equation}\label{eq:block}
 \Delta(u,v)=(U(u),V(v)),
 \qquad (u,v)\in S_{A_1\oplus_p B_1}.
\end{equation}
The operator $U\oplus V$ is the unique surjective real-linear isometric extension of $\Delta$.
\end{theorem}

We emphasize absolutely continuous and Lipschitz functions.
The main theorem gives real-linear extensions for $1<p<\infty$ and identifies the outer exponent.
Standard representations of the component isometries then yield explicit integral formulas.
The endpoint conclusions are obtained separately.
Adjoining an atom identifies the real absolutely continuous space at $p=1$ with a real $L^1$-space, so Tan's theorem gives the Mazur--Ulam property~\cite{Tan}.
For Lipschitz functions, the endpoints follow from continuous-function extension theorems and the $\ell^1$-sum theorem of Wang and Orihara~\cite{WangContinuous,WangOrihara}.
The preceding endpoint conclusions are consistent with the complex-valued Lipschitz results of Hirota and Miura~\cite{HirotaMiuraLip} and the related results in~\cite{CabezasEtAl}.
Atom-adjoining identifications also appear in Lin's work on positive unit spheres~\cite{LinAC,LinSobolev}; the conclusions here concern real-linear extensions from the full unit spheres.

For completeness, we also discuss scalar- and vector-valued differentiable functions through the identification
\[
 f\longmapsto
 \bigl((f^{(j)}(x_0))_{j=0}^{n-1},f^{(n)}\bigr).
\]
For the norms specified below, the displayed identification represents the spaces as direct sums involving $C(I,E)$.
For each fixed exponent $1<p<\infty$, these extension assertions also follow from Wang's theorem~\cite[Theorem 1.8]{WangLp}.
The same observation applies to the fixed-exponent Lipschitz assertion, since $L^\infty(I,\K)$ is isometrically a scalar continuous-function space.
The function-space applications illustrate the abstract theorem and its representations.

The paper is organized as follows.
Section~2 introduces definitions and notation and establishes preliminary results.
Section 3 identifies the components and establishes the main theorem.
Section~4 develops applications, including the absolutely continuous $\AC(I)$, Lipschitz $\Lip(I)$, continuous-functions $C(K,E)$ on compact Hausdorff $K$ and differentiable functions $C^n(I,E)$.

\section{Notation and the diametral hypothesis}
Throughout this paper, all Banach spaces are considered as real Banach spaces.
When a space is written over $\C$, we retain its complex-valued elements and its norm but restrict scalar multiplication to real scalars for this purpose.
Unless explicitly qualified as complex-linear or $\K$-linear, linearity means real linearity.
For a nonzero Banach space $E$, write
\[
 B_E=\{x\in E:\norm{x}\leq1\},\qquad
 S_E=\{x\in E:\norm{x}=1\}.
\]
For $x\in S_E$, define its \textit{diametral set} by
\[
 \diamset_E(x)=\{y\in S_E:\norm{x-y}=2\}.
\]
Notice that $-x\in\diamset_E(x)$ for every $x\in S_E$.
We give the definition of condition \textup{(D)}.

\begin{definition}\label{def:H}
A nonzero Banach space $B$ satisfies condition \textup{(D)} if $\diamset_B(b)$ contains at least two distinct points for every $b\in S_B$.
\end{definition}

\begin{proposition}\label{prop:Dsums}
If $E,F$ are nonzero Banach spaces, both $E\oplus_1F$ and $E\oplus_\infty F$ satisfy condition \textup{(D)}.
\end{proposition}
\begin{proof}
Let $(e,f)\in S_{E\oplus_1F}$.
If $\norm{e},\norm{f}>0$, the distinct unit vectors $(-e/\norm{e},0)$ and $(0,-f/\norm{f})$ are both at distance $\norm{e}+\norm{f}+1=2$ from $(e,f)$.
If $\norm{f}=0$, use $(-e,0)$ and $(0,v)$ for any $v\in S_F$.
If $\norm{e}=0$, use $(0,-f)$ and $(u,0)$ for any $u\in S_E$.

For $(e,f)\in S_{E\oplus_\infty F}$, at least one coordinate has norm one. If $\|e\|=1$, choose distinct $v_1,v_2\in B_F$.
Then $(-e,v_1)$ and $(-e,v_2)$ are distinct unit vectors at distance two from $(e,f)$. If $\|f\|=1$, argue in the other coordinate.
\end{proof}

\begin{example}
Condition \textup{(D)} holds for $\ell_1(\Gamma)$, $c_0(\Gamma)$, and $\ell_\infty(\Gamma)$ when $\Gamma$ contains at least two elements: split off one coordinate and apply Proposition~\ref{prop:Dsums}.
Condition \textup{(D)} also holds for $C(K,E)$ when $K$ is compact Hausdorff with at least two points and $E\ne\{0\}$; see Section 4.
\end{example}

Notice that a nonzero strictly convex Banach space $E$ does not satisfy \textup{(D)}, since $\mathcal D_E(x)=\{-x\}$ for every $x\in S_E$.
The direct-sum space $A\oplus_p B$ is endowed with
\[
\norm{(a,b)}_{A\oplus_pB}=
\begin{cases}
\bigl(\norm{a}_A^p+\norm{b}_B^p\bigr)^{1/p}\quad&(1\le p<\infty)\\
\max\{\|a\|_A,\|b\|_B\}\quad &(p=\infty)
\end{cases}.
\]
We use the following standard consequence of Mankiewicz's theorem.

\begin{theorem}[Mankiewicz \cite{Mankiewicz}]\label{thm:ball}
Let $E,F$ be real normed spaces. Every surjective isometry $G:B_E\to B_F$ with $G(0)=0$ is the restriction of a unique surjective real-linear isometry $\widetilde G:E\to F$.
\end{theorem}

In particular, closed unit balls are convex sets with nonempty norm interior.
Mankiewicz's theorem therefore gives an affine isometry of the ambient spaces extending the ball isometry.
The range of the affine extension contains the target unit ball; the linear part is consequently onto.
Fixing zero makes the affine extension linear.
Uniqueness follows because every vector is a positive scalar multiple of a vector in the ball.

Recall that a real normed space $A$ is strictly convex if $\norm{(x+y)/2}<1$ whenever $x,y\in S_A$ are distinct.
Equivalently, for nonzero $x,y\in A$, equality $\norm{x+y}=\norm{x}+\norm{y}$ implies
$x/\norm{x}=y/\norm{y}$. 
We use this characterization below.
At the end of this section, we recall a well-known fact.

\begin{lemma}\label{lem:radialunique}
Let $X,Y$ be nonzero real or complex normed spaces and let $\Delta:S_X\to S_Y$ be any map.
There is at most one real-linear operator $\widetilde{\Delta}:X\to Y$ extending $\Delta$.
Any such real-linear extension must satisfy
\begin{equation}\label{eq:radial_for}
 \widetilde{\Delta}(0)=0,\qquad
 \widetilde{\Delta}(x)=\norm{x}\Delta(x/\norm{x})\quad(x\ne0).
\end{equation}
If $X,Y$ are complex and this real-linear extension exists, then $\widetilde{\Delta}$ is complex-linear if and only if $\Delta(ix)=i\Delta(x)$ for all $x\in S_X$.
\end{lemma}

The uniqueness in Lemma~\ref{lem:radialunique} requires no isometry assumption.
Also, this lemma does not assert existence.
For complex spaces, identity and conjugation give different maps on the full unit sphere, so they do not contradict uniqueness of the real-linear extension of a fixed map.

\section{The main extension theorem}
\begin{lemma}\label{lem:singleton}
Let $A$ be a nonzero strictly convex Banach space, let $B$ satisfy \textup{(D)}, and let $X=A\oplus_p B$, where $1<p<\infty$.
Then
\[
P_X:=\{x\in S_X:\diamset_X(x)\text{ is a singleton}\}=S_A\times\{0\}.
\]
For $a\in S_A$, one has $\diamset_X\big((a,0)\big)=\big\{(-a,0)\big\}$.
\end{lemma}
\begin{proof}
Suppose that $a\in S_A$ and $(c,d)\in S_X$. Put $t=\norm{c}_A$.
The triangle inequality gives
\begin{align*}
\norm{(a,0)-(c,d)}_X^p&= \norm{a-c}_A^p + \norm{d}_B^p\\
&= \norm{a-c}_A^p + 1-t^p\\
&\le(\norm{a}_A + \norm{c}_A)^p + 1-t^p\\
&=(1+t)^p+1-t^p.
\end{align*}
We observe that, for $1<p<\infty$, the function $g_p(t)=(1+t)^p+1-t^p$ is strictly increasing on $[0,1]$, since $g_p'(t)=p((1+t)^{p-1}-t^{p-1})>0$.
Consequently, distance two forces $t=1$, and hence $d=0$.
It also forces $\norm{a-c}_A=2$ with $c\in S_A$.
Strict convexity of $A$ implies $c=-a$.
That is, $(c,d)=(-a,0)$, and hence $\diamset_X((a,0)) = \{(-a,0)\}$, proving that $S_A\times\{0\}\subseteq P_X$.

Conversely, let $(a,b)\in S_X$ with $r=\norm{b}_B\neq0$.
By \textup{(D)}, choose distinct $v_1,v_2\in S_B$ with $\norm{b/r-v_1}_B=\norm{b/r-v_2}_B=2$.
For $k=1,2$, the two distinct points $(-a,rv_k)$ belong to $S_X$, since
\[
\norm{(-a,rv_k)}_X^p
=\norm{a}_A^p+r^p\norm{v_k}_B^p
=\norm{a}_A^p+r^p
=\norm{a}_A^p+\norm{b}_B^p=1.
\]
Observe that 
\begin{align*}
\norm{(a,b)-(-a,rv_k)}_X^p&=\norm{2a}^p_A + \norm{b-rv_k}^p_B\\
&=2^p\norm{a}^p_A+r^p\norm{b/r - v_k}^p_B\\
&=2^p\norm{a}_A^p+2^pr^p\\
&=2^p(\norm{a}^p_A+\norm{b}^p_B)=2^p.
\end{align*}
Thus, $\diamset_X\big((a,b)\big)$ with $\norm{b}_B\neq0$ is not a singleton.
Equivalently, $\mathcal D_X((a,b))$ being a singleton implies $b=0$, proving that $P_X\subseteq S_A\times\{0\}$.
Therefore, $P_X=S_A\times\{0\}$.
\end{proof}

Lemma~\ref{lem:singleton} identifies the first component unit sphere using distances on the unit sphere. 
We next calculate distances to $P_X$ to identify the second component unit sphere and the direct-sum exponent.

\begin{lemma}\label{lem:profile}
Under the hypotheses of Lemma~\ref{lem:singleton}, let $(a,b)\in S_X$.
Then
\begin{align}
 \max_{c\in S_A}\norm{(a,b)-(c,0)}_X^p
   &=(1+\norm{a}_A)^p+1-\norm{a}_A^p,\label{eq:max}\\
 \min_{c\in S_A}\norm{(a,b)-(c,0)}_X^p
   &=(1-\norm{a}_A)^p+1-\norm{a}_A^p.\label{eq:min}
\end{align}
If $\norm{a}_A\neq0$, the unique maximizer in \eqref{eq:max} is $c=-a/\norm{a}_A$.
\end{lemma}
\begin{proof}
For $c\in S_A$, the triangle inequality yields
\[
 1-\norm{a}_A\leq\norm{a-c}_A\leq1+\norm{a}_A,
 \qquad \norm{b}_B^p=1-\norm{a}_A^p.
\]
Consequently, we have equalities~\eqref{eq:max} and~\eqref{eq:min}. 
If $\norm{a}_A\neq 0$, equality of $\norm{a-c}_A\le 1+\norm{a}_A$ is obtained at $c=-a/\norm{a}_A$, and equality of $1-\norm{a}_A\le\norm{a-c}_A$ at $c=a/\norm{a}_A$.
Strict convexity of $A$ implies that equality $\norm{a-c}_A=\norm{a}_A+\norm{c}_A$ is possible only when $a$ and $-c$ are positively collinear, and thus the maximizer is uniquely $c=-a/\norm{a}_A$.
If $\norm{a}_A=0$, every distance in question equals $2^{1/p}$.
\end{proof}

\begin{lemma}\label{lem:exponent}
For $i=1,2$, let $A_i$ be nonzero and strictly convex, let $B_i$ satisfy \textup{(D)}, and let $1<p_i<\infty$. Set
$X_i=A_i\oplus_{p_i}B_i$. 
If $\Delta:S_{X_1}\to S_{X_2}$ is a surjective isometry, then $p_1=p_2$.
Moreover,
\[
 \Delta(S_{A_1}\times\{0\})=S_{A_2}\times\{0\},\qquad
 \Delta(\{0\}\times S_{B_1})=\{0\}\times S_{B_2}.
\]
\end{lemma}
\begin{proof}
Since $\Delta:S_{X_1}\to S_{X_2}$ is a surjective isometry, for every $x\in S_{X_1}$ we have $\Delta\bigl(\mathcal D_{X_1}(x)\bigr)=\mathcal D_{X_2}(\Delta(x))$.
More precisely, preservation of distances gives one inclusion, and surjectivity gives the reverse inclusion.
Thus $\mathcal D_{X_1}(x)$ is a singleton if and only if $\mathcal D_{X_2}(\Delta(x))$ is a singleton.
By Lemma~3.1,
\[
 P_{X_i}=S_{A_i}\times\{0\}\qquad(i=1,2),
\]
and consequently $\Delta(S_{A_1}\times\{0\})=S_{A_2}\times\{0\}$.

For $i=1,2$, define
\[
 Q_{X_i}
 =\left\{x\in S_{X_i}:
   \text{the function }z\mapsto\|x-z\|_{X_i}
   \text{ is constant on }P_{X_i}\right\}.
\]
Because $\Delta$ preserves distances and maps $P_{X_1}$ onto $P_{X_2}$, it also maps $Q_{X_1}$ onto $Q_{X_2}$.

We claim that $Q_{X_i}=\{0\}\times S_{B_i}$.
To see this, let $(a,b)\in S_{X_i}$.
By Lemma~3.2,
\[
 \max_{c\in S_{A_i}}\|(a,b)-(c,0)\|_{X_i}^{p_i}
 =(1+\|a\|_{A_i})^{p_i}+1-\|a\|_{A_i}^{p_i},
\]
whereas
\[
 \min_{c\in S_{A_i}}\|(a,b)-(c,0)\|_{X_i}^{p_i}
 =(1-\|a\|_{A_i})^{p_i}+1-\|a\|_{A_i}^{p_i}.
\]
These maximum and minimum agree if and only if $\|a\|_{A_i}=0$.
Therefore, the distance to points of $P_{X_i}$ is constant exactly when $a=0$.
Since $(a,b)\in S_{X_i}$, $a=0$ implies $\|b\|_{B_i}=1$, proving the claim.
It follows that
\[
 \Delta(\{0\}\times S_{B_1})=\{0\}\times S_{B_2}.
\]

Finally, choose $a\in S_{A_1}$ and $b\in S_{B_1}$, which is possible because both spaces are nonzero.
Their images have the form
\[
 \Delta(a,0)=(a',0),\qquad
 \Delta(0,b)=(0,b'),
\]
with $a'\in S_{A_2}$ and $b'\in S_{B_2}$.
Preservation of their distance gives
\[
2^{1/p_1}=\|(a,0)-(0,b)\|_{X_1}=\|(a',0)-(0,b')\|_{X_2}=2^{1/p_2}.
\]
Because $p\mapsto 2^{1/p}$ is strictly decreasing, we conclude that $p_1=p_2$, completing the proof.
\end{proof}

\begin{proof}[\textbf{Proof of Theorem~\ref{thm:main}:}]
($\Longrightarrow$)
Assume that a surjective isometry $\Delta:S_{A_1\oplus_{p_1}B_1}\to S_{A_2\oplus_{p_2}B_2}$ exists.
Lemma~\ref{lem:exponent} gives $p_1=p_2=:p$ and $\Delta(S_{A_1}\times\{0\}) = S_{A_2}\times \{0\}$.
We write
\[
 \Delta(a_0,0)=\big(\phi(a_0),0\big),\qquad a_0\in S_{A_1},
\]
where $\phi:S_{A_1}\to S_{A_2}$ is a surjective isometry.
The last assertion of Lemma~\ref{lem:singleton} gives
\begin{equation}\label{eq:odd}
 \phi(-a_0)=-\phi(a_0),\qquad a_0\in S_{A_1}.
\end{equation}
Let $(a,b)\in S_{A_1\oplus_{p_1}B_1}$, and write $\Delta(a,b)=(a',b')$.
By equality~\eqref{eq:max}, the maximum distance to the first summand sphere determines $\norm{a}$ uniquely.
It follows that
\begin{equation}\label{eq:norm-preserve}
 \norm{a'}=\norm{a},\qquad \norm{b'}=\norm{b}.
\end{equation}
If $\norm{a}\neq 0$, by Lemma~\ref{lem:profile}, the unique maximizer $-a/\norm{a}$ is carried onto the unique maximizer $-a'/\norm{a'}$ for $(a',b')$.
Consequently, equalities~\eqref{eq:norm-preserve} yield
\[
 \phi(-a/\norm{a})=-a'/\norm{a}.
\]
Together with equality~\eqref{eq:odd}, this proves that the first coordinate of $\Delta(a,b)$ is $U_0(a)$, where
\begin{equation}\label{eq:radial}
 U_0(0)=0,\qquad
 U_0(a)=\norm{a}\phi(a/\norm{a}),\qquad 0\ne a\in B_{A_1}.
\end{equation}
For $a=0$ the same conclusion follows from equalities~\eqref{eq:norm-preserve}.
The definition~\eqref{eq:radial} shows that $U_0:B_{A_1}\to B_{A_2}$ is a norm-preserving bijection.
Its inverse is defined radially from $\phi^{-1}$, and the first coordinate of $\Delta^{-1}(c,d)$ is $U_0^{-1}(c)$.

We now extend $U_0$ to the whole space $A_1$.
Take $a,c\in B_{A_1}$.
Fix $w\in S_{B_1}$, and write
\[
 \Delta\big(a,(1-\norm{a}^p)^{1/p}w\big)=\big(U_0(a),v\big),\qquad
 \Delta\big(c,(1-\norm{c}^p)^{1/p}w\big)=\big(U_0(c),z\big).
\]
By equalities~\eqref{eq:norm-preserve}, $\norm{v}=(1-\norm{a}^p)^{1/p}$ and $\norm{z}=(1-\norm{c}^p)^{1/p}$.
Consequently,
\begin{align*}
 \norm{U_0a-U_0c}^p+\abs{(1-\norm{a}^p)^{1/p}-(1-\norm{c}^p)^{1/p}}^p
 &\leq\norm{U_0(a)-U_0(c)}^p+\norm{v-z}^p\\
 &=\norm{\big(a,(1-\norm{a}^p)^{1/p}w\big)-\big(c,(1-\norm{c}^p)^{1/p}w\big)}^p\\
 &=\norm{a-c}^p+\abs{(1-\norm{a}^p)^{1/p}-(1-\norm{c}^p)^{1/p}}^p.
\end{align*}
Thus $U_0$ is nonexpansive, namely, $\norm{U_0(a) - U_0(c)}\le\norm{a-c}$.
Applying precisely this argument to $\Delta^{-1}$ and a fixed unit vector in $B_2$ shows that $U_0^{-1}$ is nonexpansive. 
Hence $U_0$ is a surjective isometry between the closed balls. 
Theorem~\ref{thm:ball} provides a surjective real-linear isometry $U:A_1\to A_2$ extending $U_0$.

We next show that the second coordinate of $\Delta(a,b)$ is determined solely by $b$.
Let $(a,b),(c,b)\in S_{A_1\oplus_pB_1}$ and write their images as $\big(U(a),v\big)$ and $\big(U(c),z\big)$.
Since $U$ is an isometry, we have
\begin{align*}
\norm{a-c}^p&=\norm{(a,b) - (c,b)}^p\\
&=\norm{\big(U(a),v\big)- \big(U(c),z\big)}^p\\
&=\norm{U(a)-U(c)}^p+\norm{v-z}^p\\
&=\norm{a-c}^p+\norm{v-z}^p.
\end{align*}
Therefore $v=z$.
For every $b\in B_{B_1}$, choose $a\in A_1$ of norm $(1-\norm{b}^p)^{1/p}$.
We have thus defined a map $V_0:B_{B_1}\to B_{B_2}$ such that
\begin{equation}\label{eq:V0}
 \Delta(a,b)=\big(U(a),V_0(b)\big),
 \qquad(a,b)\in S_{A_1\oplus_{p}B_1}.
\end{equation}
Equation~\eqref{eq:norm-preserve} implies $V_0(0)=0$.

We now extend $V_0$ to the whole space $B_1$, and conclude the required statement.
For $b,d\in B_{B_1}$ choose $a,c$ so that $(a,b),(c,d)\in S_{A_1\oplus_pB_1}$.
Equation~\eqref{eq:V0} and the isometry of $\Delta$ give
\[
 \norm{a-c}^p+\norm{b-d}^p
 =\norm{U(a)-U(c)}^p+\norm{V_0(b)-V_0(d)}^p.
\]
Since $U$ is an isometry, the preceding equality yields $\norm{V_0(b)-V_0(d)}=\norm{b-d}$.
For $z\in B_{B_2}$ choose $c\in A_2$ such that $(c,z)$ lies on the target sphere.
A preimage under $\Delta$ shows that $z$ is in the range of $V_0$.
Hence $V_0$ is a surjective ball isometry, and we know that $V_0(0)=0$ from the previous paragraph.
Theorem~\ref{thm:ball} supplies a surjective real-linear isometry $V:B_1\to B_2$ extending $V_0$ such that
$$
 \Delta(a,b)=\big(U(a),V(b)\big),
 \qquad (a,b)\in S_{A_1\oplus_pB_1}.
$$
This proves formula~\eqref{eq:block}.

($\Longleftarrow$) Conversely, if $p_1=p_2=:p$ and $U,V$ are surjective real-linear isometries, as stated, then
$$\|(U\oplus V)(a,b)\|^p_{A_2\oplus_p B_2} = \norm{U(a)}^p_{A_2}+\norm{V(b)}^p_{B_2} = \norm{a}^p_{A_1}+\norm{b}^p_{B_1} = \norm{(a,b)}^p_{A_1\oplus_p B_1},$$
showing that $U\oplus V:A_1\oplus_pB_1\to A_2\oplus_pB_2$ is a surjective real-linear isometry.
Immediately, $(U\oplus V)|_{S_{A_1\oplus_p B_1}}$ is a surjective isometry. 
Finally, a real-linear extension of a sphere map is determined on each $x\neq0$ by radial formula~\eqref{eq:radial_for}.
Thus, Lemma~\ref{lem:radialunique} proves all uniqueness claims.
\end{proof}

\begin{remark}
Theorem~\ref{thm:main} does not assume finite dimensionality, smoothness, or a prior solution of Tingley's problem for $A_i$ or $B_i$. 
Its conclusion is nevertheless a theorem about pairs of spaces with the specified decompositions; it does not assert the Mazur--Ulam property with an arbitrary target Banach space.
\end{remark}

\begin{example}\label{ex:abstract}
Let $A$ be a nonzero strictly convex Banach space and let
$E,F$ be arbitrary nonzero Banach spaces. By
Proposition~\ref{prop:Dsums}, Theorem~\ref{thm:main} applies
to both
\[
 A\oplus_p(E\oplus_1F),\qquad
 A\oplus_p(E\oplus_\infty F),\qquad 1<p<\infty.
\]
Examples include $\R\oplus_p\ell_1^3$ and
$H\oplus_p(\ell_1\oplus_\infty L^2[0,1])$, where $H$ is a
nonzero Hilbert space. No strict convexity is required of
$E$ or $F$.
\end{example}

\section{Applications to function spaces}\label{sec:applications}

We first apply Theorem~\ref{thm:main} to absolutely continuous
functions whose norms combine evaluation at a fixed point with
the $L^1$-norm of the derivative. We then record shorter
applications to Lipschitz and differentiable functions, specifying
their relation to existing extension theorems.

Throughout this section, $I$ denotes a nondegenerate compact
interval, $x_0\in I$, and $\K\in\{\R,\C\}$. All extension
assertions concern the underlying real Banach spaces. For
$s,t\geq0$, put
\[
 N_p(s,t)=(s^p+t^p)^{1/p}\quad(1\leq p<\infty),
 \qquad N_\infty(s,t)=\max\{s,t\}.
\]
Integrals with reversed endpoints are understood as oriented
integrals. All $L^q$-spaces on intervals use the Lebesgue measure.

\subsection{Absolutely continuous functions}

Let $\AC^p(I,x_0;\K)$ denote $\AC(I,\K)$ equipped with
\begin{equation}\label{eq:ACnorm}
 \|f\|_{\AC,p,x_0}
 =N_p\bigl(|f(x_0)|,\|f'\|_1\bigr),
 \qquad 1\leq p\leq\infty.
\end{equation}
Here $p$ combines the two displayed quantities; the derivative belongs to $L^1(I,\K)$ for every $p$. 
The fundamental theorem for absolutely continuous functions gives the surjective $\K$-linear isometry
\begin{equation}\label{eq:ACdecomp}
 J_{\AC}:\AC^p(I,x_0;\K)\to\K\oplus_p L^1(I,\K),\qquad J_{\AC}(f)=(f(x_0),f'),
\end{equation}
whose inverse is
\begin{equation}\label{eq:primitive}
 J_{\AC}^{-1}(c,u)(x)=c+\int_{x_0}^x u(t)\,dt.
\end{equation}
In particular, $\AC^p(I,x_0;\K)$ is complete for all $1\le p\le\infty$.

\begin{lemma}\label{lem:Lspaces}
The spaces $L^1(I,\K)$ and $L^\infty(I,\K)$ satisfy condition~\textup{(D)}.
\end{lemma}
\begin{proof}
Partition $I$ into measurable sets $E,F$ of positive measure.
Restriction gives
\[
 L^1(I,\K)\cong L^1(E,\K)\oplus_1L^1(F,\K),
 \qquad
 L^\infty(I,\K)\cong
 L^\infty(E,\K)\oplus_\infty L^\infty(F,\K).
\]
Every component is nonzero, so Proposition~\ref{prop:Dsums} applies to both decompositions.
\end{proof}

\begin{proposition}\label{prop:AC}
Fix $\K\in\{\R,\C\}$. 
For $i=1,2$, let $I_i$ be a nondegenerate compact interval, $x_i\in I_i$, and $1<p_i<\infty$. 
Then there exists a surjective isometry $\Delta:S_{\AC^{p_1}(I_1,x_1;\K)}\to S_{\AC^{p_2}(I_2,x_2;\K)}$ if and only if $p_1=p_2$ and there exist surjective real-linear isometries $U:\K\to\K$ and $V:L^1(I_1,\K)\to L^1(I_2,\K)$.
Every such isometry $\Delta$ extends uniquely to a surjective real-linear isometry $\widetilde{\Delta}$ of the form
\begin{equation}\label{eq:ACform}
\widetilde{\Delta}(f)(x)=U\big(f(x_1)\big)+\int_{x_2}^xV(f')(t)\,dt,
 \qquad x\in I_2,
\end{equation}
where $U$ and $V$ are uniquely determined by $\Delta$.
\end{proposition}
\begin{proof}
The underlying real space of $\K$ is strictly convex, and the second components in \eqref{eq:ACdecomp} satisfy
\textup{(D)} by Lemma~\ref{lem:Lspaces}. 
Conjugating $\Delta$ by the restrictions of the maps $J_{\AC}$ to the unit spheres and applying Theorem~\ref{thm:main}
gives $p_1=p_2$ and the unique component isometries $U,V$.
Formula~\eqref{eq:primitive} then gives~\eqref{eq:ACform}.
Conversely, since $L^1(I_1,\mathbb K)$ and $L^1(I_2,\mathbb K)$ are surjectively real-linearly isometric, the existence assertion follows from Theorem~\ref{thm:main}.
\end{proof}

The following representation makes the derivative component
in \eqref{eq:ACform} explicit. Put $\kappa_0(z)=z$ and,
for complex scalars, $\kappa_1(z)=\overline z$. A measurable
bijection modulo null sets means a bimeasurable bijection
between full-measure measurable subsets such that both the
bijection and its inverse preserve null sets.

\begin{corollary}\label{cor:ACexplicit}
Under the hypotheses of Proposition~\ref{prop:AC}, every
extension has the form
\begin{equation}\label{eq:ACexplicit}
 \widetilde{\Delta}(f)(x)=\omega\kappa_{\varepsilon_0}\Big(f(x_1)\Big)
 +\int_{x_2}^x\rho(t)h(t)\kappa_{\varepsilon(t)}\Big(f'(\tau(t))\Big)\,dt.
\end{equation}
Here $\tau:I_2\to I_1$ is a measurable bijection modulo null sets, and $\rho$ is measurable, positive and finite almost everywhere, with
\begin{equation}\label{eq:ACdensity}
 \int_{I_2}\rho(t)g\big(\tau(t)\big)\,dt
 =\int_{I_1}g(s)\,ds
\end{equation}
for every nonnegative measurable function $g$ on $I_1$.
Also, $|\omega|=1$, $\varepsilon_0\in\{0,1\}$, and $h:I_2\to\T$ and $\varepsilon:I_2\to\{0,1\}$ are measurable, where $\T=\{z\in\C:|z|=1\}$.
When $\mathbb K=\mathbb R$, we require $\omega\in\{-1,1\}$, $h(t)\in\{-1,1\}$ almost everywhere, and $\varepsilon_0=\varepsilon=0$.
\end{corollary}

\begin{proof}
By Proposition~\ref{prop:AC}, we have
\[
\widetilde{\Delta}(f)(x) = U(f(x_1)) + \int_{x_2}^{x} V(f')(t)\,dt,
\]
where $U:\mathbb K\to\mathbb K$ and $V:L^1(I_1,\mathbb K)\to L^1(I_2,\mathbb K)$ are surjective real-linear isometries.
Every surjective real-linear isometry of $\mathbb K$ has the form
\[
U(z)=\omega\kappa_{\varepsilon_0}(z),
\]
where $|\omega|=1$, $\varepsilon_0\in\{0,1\}$, $\kappa_0(z)=z$, and $\kappa_1(z)=\overline z$.
In the real case, $\omega\in\{-1,1\}$ and $\varepsilon_0=0$.

We next describe $V$.
Regard $\K$ as a real Banach space, namely $\R$ or the Euclidean plane $\R^2$.
In either case, $\K$ is separable and strictly convex, and hence cannot be expressed as an $\ell^1$-direct sum of two nonzero Banach spaces. 
Therefore Sourour's representation theorem~\cite{Sourour} applies to $V$.
The two intervals may first be identified by an affine bijection, together with the corresponding normalization of the $L^1$-norm.
Sourour’s theorem yields an isomorphism between the underlying measure algebras, a positive measurable weight, and a measurable family of surjective real-linear isometries of $\K$.

Since $I_1$ and $I_2$, equipped with the Lebesgue measure, are standard nonatomic finite measure spaces, the measure-algebra isomorphism can be represented by a measurable bijection $\tau:I_2\to I_1$ modulo null sets~\cite[344C]{Fremlin}.
Consequently, there exist a measurable function $\rho:I_2\to(0,\infty)$, finite almost everywhere, and a measurable family of surjective real-linear isometries $T_t:\K\to\K$ such that
\[
V(u)(t)=\rho(t)T_t\big(u(\tau(t))\big)
\]
for every $u\in L^1(I_1,\K)$ and almost every $t\in I_2$.

Each $T_t$ is multiplication by a sign when $\K=\R$.
When $\K=\C\simeq\R^2$, regarded as a real Banach space, every such isometry is either multiplication by a unimodular scalar or multiplication by a unimodular scalar followed by complex conjugation.
Thus there exist measurable functions $h:I_2\to\T$ and $\varepsilon:I_2\to\{0,1\}$ such that
\[
T_t(z)
=
h(t)\kappa_{\varepsilon(t)}(z)
\]
for almost every $t\in I_2$.
In the real case, $h(t)\in\{-1,1\}$ almost everywhere and $\varepsilon=0$.
Hence
\[
V(u)(t)
=
\rho(t)h(t)
\kappa_{\varepsilon(t)}
\bigl(u(\tau(t))\bigr)
\quad\text{a.e.}
\]

It remains to verify equality~\eqref{eq:ACdensity}.
Let $g$ be a nonnegative integrable function on $I_1$, and choose a fixed $e\in\K$ with $|e|=1$. 
Applying the preceding representation to $u(s)=g(s)e$, and using the fact that $V$ and each $T_t$ are isometries, we obtain
\[
\begin{aligned}
\int_{I_1}g(s)\,ds=\|u\|_1=\|V(u)\|_1=\int_{I_2}\rho(t)
  \left|
  T_t\bigl(g(\tau(t))e\bigr)
  \right|\,dt=\int_{I_2}\rho(t)g(\tau(t))\,dt.
\end{aligned}
\]
By applying this identity to truncations and using the monotone convergence theorem, the equality extends to every nonnegative measurable function $g$ on $I_1$. 
This proves equality~\eqref{eq:ACdensity}.
Finally, substituting the representations of $U$ and $V$ into formula~\eqref{eq:ACform} yields~\eqref{eq:ACexplicit}, completing the proof.
\end{proof}

\begin{remark}\label{rem:ACcomparison}
Proposition~\ref{prop:AC} uses condition~\textup{(D)} of the $L^1$-component directly. Wang's~\cite[Theorem~1.8]{WangLp} is formulated for specified $C_0(\Omega,E)$-type components and a common exponent.
The applicability of Wang’s theorem to an $L^1$-component does not follow merely from the identity $L^1(I,\K)\cong C_0\big(\{*\},L^1(I,\K)\big)$:
the singleton component in that theorem is required to be strictly convex, whereas $L^1(I,\K)$ is not strictly convex.
The representation in Corollary~\ref{cor:ACexplicit} uses the existing classification of linear $L^1$-isometries; the extension assertion is supplied by Theorem~\ref{thm:main}.
\end{remark}

The following proposition establishes that $\AC^1(I,x_0 ;\R)$ has the Mazur--Ulam property.
\begin{proposition}\label{prop:ACone}
The real Banach space $\AC^1(I,x_0;\R)$ has the Mazur--Ulam property.
\end{proposition}
\begin{proof}
Let $\widetilde I=I\sqcup\{*\}$ and $\mu=m+\delta_*$, where $m$ is the Lebesgue measure and $\delta_\ast$ is the Dirac measure at $\{\ast\}$.
The map
\[
Q(f):=f(x_0)\chi_{\{\ast\}}+f'\chi_I,\qquad f\in \AC^1(I,x_0;\R)
\]
is a surjective real-linear isometry onto $L^1(\widetilde I,\mu;\R)$ by formula \eqref{eq:primitive}.
Since $\mu$ is finite, Tan's theorem~\cite{Tan} applies.
Let $\Delta:S_{\AC_1(I,x_0;\R)}\to S_Y$ be a surjective isometry, where $Y$ is a real Banach space.
Then $\Delta\circ Q^{-1}$ extends to a surjective real-linear isometry 
$$R:L^1(\widetilde I,\mu;\R)\to Y.$$
The operator $R\circ Q$ extends $\Delta$.
Uniqueness follows from Lemma~\ref{lem:radialunique}.
\end{proof}

Notice that the endpoint conclusion in Proposition~\ref{prop:ACone} is a consequence of the existing
$L^1$ extension theorem.

\subsection{Lipschitz functions}

For $f\in\Lip(I,\K)$, define
\[
 L(f):=\sup_{s\ne t}\frac{|f(s)-f(t)|}{|s-t|},
 \qquad
 \|f\|_{\Lip,p,x_0}:=N_p(|f(x_0)|,L(f)),
\]
and denote the resulting space by $\Lip^p(I,x_0;\K)$.
Since $I$ is a compact interval, $L(f)=\|f'\|_{L^\infty}$.
Consequently, the map
\begin{equation}\label{eq:Lipdecomp}
 J_{\Lip}(f):=(f(x_0),f'):\Lip^p(I,x_0;\K)\to\K\oplus_p L^\infty(I,\K)
\end{equation}
is a surjective $\K$-linear isometry, with inverse
$(c,u)\mapsto c+\int_{x_0}^{\,\cdot}u(t)\,dt$.

\begin{proposition}\label{prop:Lip}
Fix $\K\in\{\R,\C\}$. 
Suppose $1<p_i<\infty$.
Then $S_{\Lip^{p_1}(I_1,x_1;\K)}$ is surjectively isometric to $S_{\Lip^{p_2}(I_2,x_2;\K)}$ if and only if $p_1=p_2=:p$.
For every common exponent $1\leq p\leq\infty$, every surjective isometry $\Delta:S_{\Lip_p(I_1,x_1;\K)}\to S_{\Lip_p(I_2,x_2;\K)}$ extends uniquely to a surjective real-linear isometry $\widetilde{\Delta}$ of the form
\begin{equation}\label{eq:Lipform}
\widetilde{\Delta}(f)(x)=U\big(f(x_1)\big)+\int_{x_2}^xV(f')(t)\,dt,
\end{equation}
where $U:\K\to\K$ and $V:L^\infty(I_1,\K)\to L^\infty(I_2,\K)$ are surjective real-linear isometries uniquely determined by $\Delta$.
\end{proposition}
\begin{proof}
For $1<p_i<\infty$, apply Theorem~\ref{thm:main} to~\eqref{eq:Lipdecomp}, using Lemma~\ref{lem:Lspaces}.
Existence for equal exponents follows by taking $U$ to be the identity and $V(u)=u\circ\tau$, where $\tau:I_2\to I_1$ is an increasing affine bijection.

For the endpoints, write $L^\infty(I_i,\K)\cong C(K_i,\K)$ using its scalar continuous-function representation. 
For $p=1$, apply~\cite[Theorem~3.5]{WangOrihara} to $C(\{*\},G)\oplus_1 C(K_i,G)$, where $G$ is the underlying real space of $\K$. 
Neither component is isometric to real $\ell_\infty^2$: $G$ is strictly convex, and the other component is infinite-dimensional.
The component representation in~\cite[Theorems~3.3 and~3.4]{WangOrihara} therefore applies.
The finite-dimensional component must be preserved, giving \eqref{eq:Lipform}.

For $p=\infty$, the spaces identify with $C(K_i\sqcup\{*\},\K)$.
The scalar continuous-function extension results \cite{WangContinuous,MoriOzawa} give a real-linear extension.
The compact space $K_i$ has no isolated points, since the Lebesgue measure algebra has no atoms.
The real Banach--Stone theorem~\cite{FlemingJamison} consequently preserves the unique isolated point $*$, again giving \eqref{eq:Lipform}.
\end{proof}

For a common $1<p<\infty$, the extension and component representation above also follow from
\cite[Theorem~1.8]{WangLp}, since $L^\infty(I_i,\K)\cong C(K_i,\K)$.
Theorem~\ref{thm:main} additionally compares possibly different exponents.
The endpoint assertions were obtained from the cited theorems.
In particular, the complex $p=1$ extension problem with norm $|f(0)|+\|f'\|_{L^\infty}$ was studied by Hirota and Miura \cite{HirotaMiuraLip}; see also \cite{CabezasEtAl}.

\begin{remark}[Explicit Lipschitz extensions]\label{rem:Lipexplicit}
By the real Banach--Stone representation of the $L^\infty$ component and \cite[344C]{Fremlin}, formula~\eqref{eq:Lipform} can be rewritten as
\begin{equation}\label{eq:LIPexp}
 \widetilde{\Delta}(f)(x)=\omega\kappa_{\varepsilon_0}\Big(f(x_1)\Big)
 +\int_{x_2}^x h(t)
       \kappa_{\varepsilon(t)}\Big(f'(\tau(t))\Big)\,dt.
\end{equation}
The symbols $\omega,\varepsilon_0,h,\varepsilon,\tau$ have the meanings specified in Corollary~\ref{cor:ACexplicit}.
Every such choice gives a surjective real-linear isometry.
In fact, the null-set preservation in both directions implies $\|u\circ\tau\|_\infty=\|u\|_\infty$, and the pointwise scalar isometries and $\tau$ are invertible modulo null sets.
The representation~\eqref{eq:LIPexp} contains no density factor, since the derivative norm is the essential supremum norm.
\end{remark}

\subsection{Continuous and differentiable functions}
We first verify condition \textup{(D)} for the continuous-function components needed for the application: every $C(K,E)$ with $K$ compact Hausdorff, $|K|\geq2$, and $E\ne\{0\}$ satisfies \textup{(D)}. 
To verify condition \textup{(D)}, let $u\in S_{C(K,E)}$, choose $t_0$ with $\|u(t_0)\|=1$, a point $t_1\ne t_0$, and $\eta\in C(K,[0,1])$ with $\eta(t_0)=0$, $\eta(t_1)=1$.
Writing $e=u(t_0)$, the distinct functions
\[
 v_1(t)=-e,\qquad v_2(t)=-(1-\eta(t))e
\]
belong to $S_{C(K,E)}$ and satisfy $\|u-v_1\|_\infty=\|u-v_2\|_\infty=2$.

Let $E$ be a nonzero real or complex Banach space and $n\geq1$.
Equip $C^n(I,E)$ with
\begin{equation}\label{eq:anchored}
 \|f\|_{n,p,x_0}=
 \begin{cases}
 \displaystyle\left(\sum_{j=0}^{n-1}
       \|f^{(j)}(x_0)\|_E^p+\|f^{(n)}\|_\infty^p
       \right)^{1/p},&1\leq p<\infty,\\[6pt]
 \displaystyle\max\left\{
       \max_{0\leq j<n}\|f^{(j)}(x_0)\|_E,
       \|f^{(n)}\|_\infty\right\},&p=\infty,
 \end{cases}
\end{equation}
and write $C_p^n(I,x_0;E)$ for the normed space $(C^n(I,E), \norm{\cdot}_{n,p,x_0})$. 
Derivatives at endpoints are understood as one-sided derivatives.

\begin{proposition}\label{prop:jet}
For $1\leq p\leq\infty$, $n\ge1$, the map
\[
 J(f):=\Big((f^{(j)}(x_0))_{j=0}^{n-1},f^{(n)}\Big)
\]
is a surjective real-linear isometry from $C_p^n(I,x_0;E)$ onto $\ell_p^n(E)\oplus_p C(I,E)$. Its inverse is
\begin{equation}\label{eq:Taylor}
 J^{-1}(a,u)(x)=\sum_{j=0}^{n-1}
 a_j\frac{(x-x_0)^j}{j!}
 +\frac1{(n-1)!}\int_{x_0}^x(x-t)^{n-1}u(t)\,dt.
\end{equation}
In particular, $C_p^n(I,x_0;E)$ is a Banach space.
\end{proposition}
\begin{proof}
For $u\in C(I,E)$ and $m\geq1$, define the Bochner integral
\[
 R_m u(x)=\frac1{(m-1)!}
          \int_{x_0}^x(x-t)^{m-1}u(t)\,dt.
\]
The fundamental theorem of calculus gives $(R_1u)'=u$.
For $m\geq2$, differentiation under the integral sign gives $(R_m u)'=R_{m-1}u$; the boundary term vanishes.
These identities hold with one-sided derivatives at the endpoints, and $R_m u(x_0)=0$ for every $m\geq1$.
Thus the function $F$ on the right of~\eqref{eq:Taylor} belongs to $C^n(I,E)$ and satisfies
\[
 F^{(k)}(x)=\sum_{j=k}^{n-1}
 a_j\frac{(x-x_0)^{j-k}}{(j-k)!}+R_{n-k}u(x)
 \quad(0\leq k<n),
 \qquad F^{(n)}=u.
\]
Consequently, $F^{(k)}(x_0)=a_k$ and $J(F)=(a,u)$.
Hence, $J$ is surjective.
If $J(f)=0$, successive applications of
\[
 f^{(k)}(x)=f^{(k)}(x_0)
             +\int_{x_0}^x f^{(k+1)}(t)\,dt
\]
for $k=n-1,\ldots,0$ give $f=0$, proving injectivity.
Linearity is immediate, and the norm~\eqref{eq:anchored} gives $\|J(f)\|=\|f\|_{n,p,x_0}$, including the maximum norm when $p=\infty$. 
Completeness follows from completeness of the two components.
\end{proof}

\begin{corollary}\label{cor:vector}
For $i=1,2$, let $E_i$ be nonzero strictly convex Banach
spaces, $n_i\geq1$, and $1<p_i<\infty$. 
Then there exists a surjective isometry $\Delta:S_{C_{p_1}^{n_1}(I_1,x_1;E_1)}\to S_{C_{p_2}^{n_2}(I_2,x_2;E_2)}$ if and only if $p_1=p_2=:p$ and there exist surjective real-linear isometries $U:\ell_p^{n_1}(E_1)\to\ell_p^{n_2}(E_2)$ and
$V:C(I_1,E_1)\to C(I_2,E_2)$.
Moreover, every such surjective isometry $\Delta$ extends uniquely to
\begin{equation}\label{eq:vectorT}
 \widetilde{\Delta}:=J_2^{-1}(U\oplus V)J_1,
\end{equation}
where $U$ and $V$ are uniquely determined by $\Delta$.
For a fixed scalar field $E_1=E_2=\K$, such $\Delta$ exists if and only if $(n_1,p_1)=(n_2,p_2)$.
\end{corollary}
\begin{proof}
For $1<p<\infty$, the finite $\ell^p$-sum of strictly convex spaces is strictly convex.
Apply Theorem~\ref{thm:main} through Proposition~\ref{prop:jet}, using condition~\textup{(D)} of $C(I_i,E_i)$.

In the scalar case, the real dimensions of the first components give $n_1=n_2$.
Conversely, if $(n_1,p_1)=(n_2,p_2)$, take $U$ to be the identity and define $V(u)=u\circ\tau$, where $\tau:I_2\to I_1$ is a homeomorphism.
Formula~\eqref{eq:vectorT} gives the required isometry.
\end{proof}

For a common exponent $1<p<\infty$, Corollary~\ref{cor:vector} is already covered by \cite[Theorem~1.8]{WangLp}.
More generally, Wang's Theorem 1.8 covers $A\oplus_p C(K,E)$ when $A$ is strictly convex and $|K|\geq2$, without requiring strict convexity of $E$.
The additional value-space assumptions in Wang's Theorem~1.10 concern pointwise representations.

\begin{remark}[Known endpoint consequences]\label{rem:Cend}
For $p=1$ and strictly convex $E_i$, the decomposition into $n_i$ copies of $E_i$ and one copy of $C(I_i,E_i)$ allows the application of \cite[Theorems~3.3--3.5]{WangOrihara}.
Each component is of $C_0(\Omega,E_i)$ type.
None is isometric to real $\ell_\infty^2$, and $C(I_i,E_i)$ is the unique component that is not strictly convex.
Consequently, every surjective isometry between the unit spheres extends, $n_1=n_2=n$, and the extension is
\[
 \widetilde{\Delta}(f)(x)=\sum_{j=0}^{n-1}
 U_j\Big(f^{(\pi(j))}(x_1)\Big)
 \frac{(x-x_2)^j}{j!}
 +\frac1{(n-1)!}\int_{x_2}^x(x-t)^{n-1}
                         V\big(f^{(n)}\big)(t)\,dt,
\]
where $\pi$ is a permutation and $U_j:E_1\to E_2$ and $V:C(I_1,E_1)\to C(I_2,E_2)$ are surjective real-linear isometries.
The converse follows from the decomposition.

For scalar-valued functions and $p=\infty$,
\[
 C_\infty^n(I,x_0;\K)\cong C(I\sqcup D_n,\K),
\]
where $D_n$ is a discrete set of $n$ points.
The scalar continuous-function extension theorems~\cite{WangContinuous,MoriOzawa}, followed by the real
Banach--Stone theorem, give extension and equality of the orders, since a homeomorphism preserves the isolated points. 
For general $E$, the analogous identification is with $C(I\sqcup D_n,E)$; an extension conclusion then requires a suitable vector-valued continuous-function theorem.
\end{remark}

\begin{remark}[Scalar representation]
For a common $n\geq1$ and $1\leq p\leq\infty$, the scalar extensions above can be written as
\[
 \widetilde{\Delta}(f)(x)=\sum_{j=0}^{n-1}(Ua_f)_j\frac{(x-x_2)^j}{j!}
 +\frac1{(n-1)!}\int_{x_2}^x(x-t)^{n-1}h(t)\kappa_\varepsilon\Big(f^{(n)}(\tau(t))\Big)\,dt,
 \quad a_f=(f^{(j)}(x_1))_{j=0}^{n-1}.
\]
Here $\tau:I_2\to I_1$ is a homeomorphism, $h$ is continuous with $|h|=1$, and $\varepsilon\in\{0,1\}$ is constant in the complex case. 
In the real case, $h$ is a constant sign and $\varepsilon=0$.
The real Banach--Stone representation of the continuous-function component is determined by $\tau,h$ and $\varepsilon$~\cite{FlemingJamison}.
For $p\ne2$, the finite-dimensional component has the form
\[
 (Ua)_j=\omega_j\kappa_{\varepsilon_j}(a_{\pi(j)}),
\]
with a permutation $\pi$, unimodular $\omega_j$, and independent conjugation choices; over $\R$ the factors are signs and all $\varepsilon_j=0$.
The stated representation of $U$ follows from preservation of disjoint coordinate supports for $1\leq p<\infty$, $p\ne2$, and duality for $p=\infty$. 
For $p=2$, $U$ is an arbitrary element of $O(n)$ over $\R$, or $O(2n)$ over $\C$ viewed as a real space. 
Conversely, every choice of $U$, $\tau$, $h$, and $\varepsilon$ satisfying the stated conditions determines a surjective real-linear isometry.
\end{remark}

It is worth emphasizing that, in the scalar-valued case, the differentiability order $n$ is determined by the metric structure of the unit sphere. 
More precisely, for $1<p_1,p_2<\infty$, Corollary~\ref{cor:vector} shows that the existence of a surjective isometry between the unit spheres of $C_{p_1}^{n_1}(I_1,x_1;\mathbb K)$ and $C_{p_2}^{n_2}(I_2,x_2;\mathbb K)$ forces $(n_1,p_1)=(n_2,p_2)$, without any linearity assumption on the given isometry. 
For a fixed exponent $1<p<\infty$, the preservation of the differentiability order already follows from Wang's theorem~\cite[Theorem 1.8]{WangLp} and the finite-dimensionality of the first component.
The additional feature of Corollary 4.9 is that the outer exponents are also allowed to differ initially and are necessarily equal.

\section*{Use of AI tools}
The author used ChatGPT (OpenAI) to assist with developing and revising mathematical arguments and proof drafts,
organizing the manuscript, and improving its exposition.
This assistance included the direct-sum extension arguments and the function-space applications in Sections~3 and~4.
The author independently checked the mathematical arguments and references, revised the generated material, and takes full responsibility for the final manuscript.


\begin{thebibliography}{99}

\bibitem{Banakh}
T.~Banakh,
\emph{Every $2$-dimensional Banach space has the
Mazur--Ulam property},
Linear Algebra Appl. \textbf{632} (2022), 268--280.
\href{https://doi.org/10.1016/j.laa.2021.09.020}
{doi:10.1016/j.laa.2021.09.020}.

\bibitem{CabezasEtAl}
D.~Cabezas, M.~Cueto-Avellaneda, Y.~Enami, T.~Miura, and A.~M.~Peralta,
\emph{Tingley's problem for complex Banach spaces which do not
satisfy the Hausdorff distance condition},
Banach J. Math. Anal. \textbf{17} (2023), Article 65.
\href{https://doi.org/10.1007/s43037-023-00291-0}{doi:10.1007/s43037-023-00291-0}.
Preprint: \href{https://arxiv.org/abs/2210.17131}{arXiv:2210.17131}.

\bibitem{FlemingJamison}
R.~J.~Fleming and J.~E.~Jamison,
\emph{Isometries on Banach Spaces: Function Spaces},
Chapman \& Hall/CRC Monographs and Surveys in Pure and Applied
Mathematics, vol.~129, Chapman \& Hall/CRC, Boca Raton, FL, 2003.

\bibitem{Fremlin}
D.~H.~Fremlin,
\emph{Measure Theory}, vol.~3: Measure Algebras,
Torres Fremlin, Colchester, 2002.

\bibitem{HirotaMiuraLip}
D.~Hirota and T.~Miura,
\emph{Tingley's problem for a Banach space of Lipschitz functions
on the closed unit interval},
RIMS K\^oky\^uroku Bessatsu \textbf{B93} (2023), 157--181.

\bibitem{LinAC}
M.-R.~Lin,
\emph{A variant of Tingley's problem on ordered Banach spaces of
absolutely continuous functions}, preprint (2026).
\href{https://arxiv.org/abs/2607.10685}{arXiv:2607.10685}.

\bibitem{LinSobolev}
M.-R.~Lin,
\emph{Metric rigidity in anchored Sobolev spaces on intervals},
preprint (2026).
\href{https://arxiv.org/abs/2607.27646}{arXiv:2607.27646}.

\bibitem{Mankiewicz}
P.~Mankiewicz,
\emph{On extension of isometries in normed linear spaces},
Bull. Acad. Polon. Sci. S\'er. Sci. Math. Astronom. Phys.
\textbf{20} (1972), 367--371.

\bibitem{MoriOzawa}
M.~Mori and N.~Ozawa,
\emph{Mankiewicz's theorem and the Mazur--Ulam property
for $C^*$-algebras},
Studia Math. \textbf{250} (2020), no.~3, 265--281.
\href{https://doi.org/10.4064/sm180727-14-11}
{doi:10.4064/sm180727-14-11}.

\bibitem{PeraltaSurvey}
A.~M.~Peralta,
\emph{A survey on Tingley's problem for operator algebras},
Acta Sci. Math. (Szeged) \textbf{84} (2018),
nos.~1--2, 81--123.
\href{https://doi.org/10.14232/actasm-018-255-0}
{doi:10.14232/actasm-018-255-0}.

\bibitem{PeraltaSvarc}
A.~M.~Peralta and R.~\v{S}varc,
\emph{A strengthened Kadison's transitivity theorem
for unital $JB^*$-algebras with applications to the
Mazur--Ulam property},
Anal. Math. Phys. \textbf{15} (2025), Article 91.
\href{https://doi.org/10.1007/s13324-025-01068-4}
{doi:10.1007/s13324-025-01068-4}.

\bibitem{Sourour}
A.~R.~Sourour,
\emph{The isometries of $L^p(\Omega,X)$},
J. Funct. Anal. \textbf{30} (1978), no.~2, 276--285.
\href{https://doi.org/10.1016/0022-1236(78)90075-7}
{doi:10.1016/0022-1236(78)90075-7}.

\bibitem{Tan}
D.-N.~Tan,
\emph{On extension of isometries on the unit spheres of
$L^p$-spaces for $0<p\leq1$},
Nonlinear Anal. \textbf{74} (2011), no.~18, 6981--6987.
\href{https://doi.org/10.1016/j.na.2011.07.035}{doi:10.1016/j.na.2011.07.035}.

\bibitem{Tingley}
D.~Tingley,
\emph{Isometries of the unit sphere},
Geom. Dedicata \textbf{22} (1987), no.~3, 371--378.
\href{https://doi.org/10.1007/BF00147942}{doi:10.1007/BF00147942}.

\bibitem{WangContinuous}
R.~Wang,
\emph{Isometries between the unit spheres of $C_0(\Omega)$-type spaces},
Acta Math. Sci. \textbf{14} (1994), no.~1, 82--89.

\bibitem{WangLp}
R.~Wang,
\emph{Isometries on the $\ell^p$-sum of $C_0(\Omega,E)$ type spaces},
J. Math. Sci. Univ. Tokyo \textbf{3} (1996), no.~2, 471--493.

\bibitem{WangOrihara}
R.~Wang and A.~Orihara,
\emph{Isometries on the $\ell^1$-sum of $C_0(\Omega,E)$ type spaces},
J. Math. Sci. Univ. Tokyo \textbf{2} (1995), no.~1, 131--154.


\end{thebibliography}
\end{document}